\documentclass{amsart}
\usepackage{amssymb,amsmath, amsthm,latexsym}
\usepackage{graphics}
\usepackage{psfrag}
\usepackage{amscd}
\usepackage{graphicx}
\usepackage{here}
\newcommand{\cal}[1]{\mathcal{#1}}
\theoremstyle{plain}

\newtheorem{theo}{Theorem}

\newtheorem{lemma}{Lemma}[section]

\newtheorem{proposition}[lemma]{Proposition}
\newtheorem{corollary}[lemma]{Corollary}
\theoremstyle{definition}
\newtheorem*{definition}{Definition}

\begin{document}
\title{Dynamical properties of the Weil-Petersson metric}
\author{Ursula Hamenst\"adt}
%Mathematisches Institut der Universit\"at Bonn,\\
%Beringstra\ss{}e 1, D-53115 Bonn}
\thanks %{\it e-mail address:} ursula@math.uni-bonn.de\\
{Partially supported by Sonderforschungsbereich 611.}

\date{January 26, 2008}

\begin{abstract}
Let ${\cal T}(S)$ be the
Teichm\"uller space of an oriented surface $S$ of finite type.
We discuss the action of subgroups of the mapping class
group of $S$ on the ${\rm CAT}(0)$-boundary of the completion
of ${\cal T}(S)$ with respect to the Weil-Petersson
metric. We show that the set of invariant Borel
probability measures for the 
Weil-Petersson 
flow on moduli space which are supported on a closed orbit is 
dense in the space of all ergodic invariant probability measures.
\end{abstract}

\maketitle

\section{Introduction}

For an oriented surface $S$ of genus $g\geq 0$ with 
$m\geq 0$ punctures and \emph{complexity} $3g-3+m\geq 2$ let 
${\cal T}(S)$ be the \emph{Teichm\"uller space} 
of all isotopy classes of complete hyperbolic 
metrics on $S$ of finite volume. Then ${\cal T}(S)$ is
a contractible manifold which can be 
equipped with
the \emph{Weil-Petersson metric}, an incomplete
K\"ahler metric of negative sectional curvature.

In spite of the lack of completeness, any
two points in ${\cal T}(S)$ can be
connected by a unique Weil-Petersson
geodesic which depends smoothly on
its endpoints \cite{DW03}. As a consequence, ${\cal T}(S)$ can
be completed to a \emph{Hadamard space}
$\overline{{\cal T}(S)}$, i.e. a
complete simply connected ${\rm CAT}(0)$-space which
however is not locally compact.

A Hadamard space $X$ admits a \emph{visual boundary}
$\partial X$, and the action of the isometry group of $X$ 
extends to an action on $\partial X$. 
For surfaces $S$ of complexity at most three,
the visual boundary $\partial\overline{{\cal T}(S)}$ 
of $\overline{{\cal T}(S)}$ was identified by
Brock and Masur \cite{BM08}, but for higher complexity
it is not known.
However, it follows from the work of Brock \cite{B05} that
the boundary is not locally compact.

The \emph{extended mapping class group} ${\rm Mod}(S)$ of all
isotopy classes of diffeomorphisms of $S$ acts on 
${\cal T}(S)$
as a group of isometries. Since every isometry of
${\cal T}(S)$ extends to an isometry of the completion
$\overline{{\cal T}(S)}$, the extended mapping
class group also acts isometrically on $\overline{{\cal T}(S)}$. 
This fact was used by Masur and Wolf \cite{MW02} to 
show that in fact every isometry of $\overline{{\cal T}(S)}$
is an extended mapping class.

An isometry $g$ of $\overline{{\cal T}(S)}$ is called \emph{axial}
if $g$ admits an axis, i.e. if there is a geodesic
$\gamma:\mathbb{R}\to \overline{{\cal T}(S)}$ 
and a number $\tau >0$ such that $g\gamma(t)=\gamma(t+\tau)$ for all $t$.
The endpoints $\gamma(\infty),\gamma(-\infty)$ are then
fixed points for the action of $g$ on $\partial\overline{{\cal T}(S)}$.
Every \emph{pseudo-Anosov} mapping class $g\in {\rm Mod}(S)$ 
is axial \cite{DW03}.
The \emph{limit set} $\Lambda$ of a subgroup $G$ 
of ${\rm Mod}(S)$ is 
the set of accumulation points in 
$\partial \overline{{\cal T}(S)}$ of an orbit of the
action of $G$ on $\overline{{\cal T}(S)}$.
The group $G$ is called \emph{non-elementary} if its
limit set contains at least three points. 
We show.

\begin{theo}\label{thm1}
Let $G<{\rm Mod}(S)$ be a non-elementary subgroup 
with limit set $\Lambda$ which contains a pseudo-Anosov element.
\begin{enumerate}
\item $\Lambda$ does not have isolated points, 
and the $G$-action on $\Lambda$ is minimal.
\item Pairs of fixed points of pseudo-Anosov elements are dense in 
$\Lambda\times \Lambda$.
\item There is a dense orbit for the action of $G$ on
$\Lambda\times \Lambda$.
%\item $G$ contains a free subgroup with two generators
%consisting of rank-one elements. 
\end{enumerate}
\end{theo}

There is a compactification 
of ${\cal T}(S)$ whose boundary consists of the
sphere ${\cal P\cal M\cal L}$ of projective
measured geodesic laminations on $S$. 
The mapping class group acts on ${\cal T}(S)\cup {\cal P\cal M\cal L}$
as a group of homeomorphisms.
Limit sets for
subgroups of ${\rm Mod}(S)$ in ${\cal P\cal M\cal L}$
were investigated by McCarthy and Papadopoulos \cite{MCP89}.

We also look at properties of the 
\emph{Weil-Petersson geodesic
flow} $\Phi^t$ on the quotient 
of the unit tangent bundle $T^1{\cal T}(S)$ of ${\cal T}(S)$ under the
action of the mapping class group. Even though
this quotient space $T^1{\cal M}(S)$ is non-compact and 
this flow is not everywhere defined, 
it admits many invariant
Borel probability measures. Particular such measures are measures
supported on periodic orbits. Each of these measures is ergodic.

The space of all $\Phi^t$-invariant Borel probability measures
on $T^1{\cal M}(S)$ can be equipped with the
weak$^*$-topology.
Our second result is a version of Theorem 1 for 
the Weil-Petersson geodesic flow. 

\begin{theo}
A $\Phi^t$-invariant Borel probability measure 
on $T^1{\cal M}(S)$ can
be approximated in the weak$^*$-topology by measures supported
on periodic orbits.
\end{theo}

The organization of this note is as follows. In Section 2 we
review some geometric properties
of Hadamard spaces. Section 3 explains
some geometric properties of pseudo-Anosov mapping classes.
In Section 4 we look at groups of isometries and
establish the first and the second part of Theorem 1. 
In Section 5, we complete the proof of Theorem 1 and show 
Theorem 2.

\section{Basic ${\rm CAT}(0)$-geometry}

The purpose of this section is to collect some general 
geometric
properties of ${\rm CAT}(0)$-spaces which are 
needed for the investigation of Weil-Petersson space.

A ${\rm CAT}(0)$-space is defined as follows.
A \emph{triangle} $\Delta$ in a geodesic metric space consists
of three vertices connected by three (minimal) geodesic arcs $a,b,c$. 
A \emph{comparison triangle} $\bar \Delta$ for $\Delta$ 
in the euclidean plane
is a triangle in $\mathbb{R}^2$ 
with the same side-lengths as $\Delta$. By the triangle inequality,
such a comparison triangle exists always, and it is 
unique up to isometry. For a point $x\in a\subset \Delta$
the comparison point of $x$ 
in the comparison triangle $\bar\Delta$ 
is the point on the side $\bar a$ of $\bar\Delta$ corresponding to $a$
whose distance to the endpoints of $\bar a $
coincides with the distance of $x$ to the
corresponding endpoints of $a$.

A geodesic metric space $(X,d)$ 
is called a \emph{${\rm CAT}(0)$-space}
if for every geodesic triangle $\Delta$ in $X$ with 
sides $a,b,c$ and every comparison triangle
$\bar\Delta$ 
in the euclidean plane 
with sides $\bar a,\bar b,\bar c$  
and for all $x,y\in \Delta$ and all 
comparison points 
$\bar x,\bar y\in \bar \Delta$ we have
\[d(x,y)\leq d(\bar x,\bar y).\]
A complete ${\rm CAT}(0)$-space is called a \emph{Hadamard space}.

In a Hadamard space $X$, the distance
function is convex: If $\gamma,\zeta$ are two
geodesics in $X$ parametrized on the same interval 
then the function 
$t\to d(\gamma(t),\zeta(t))$ is convex.
For two geodesics $\gamma,\zeta$ issuing from the same
point $\gamma(0)=\zeta(0)$, the \emph{Alexandrov angle} 
between $\gamma,\zeta$ is defined. If $X$ is a Riemannian
manifold of non-positive curvature, then this angle
coincides with the angle between the tangents of 
$\gamma,\zeta$ at $\gamma(0)$ (see \cite{BH99}).

For a fixed point $x\in X$,
the \emph{visual boundary} $\partial X$ of $X$
is defined to be the space
of all geodesic rays issuing from $x$ equipped with the
topology of uniform convergence on bounded sets.
This definition is independent of the choice of $x$.
We denote the point in $\partial X$ defined by 
a geodesic ray $\gamma:[0,\infty)\to X$ by $\gamma(\infty)$.
We also say that $\gamma$ 
\emph{connects} $x$ to 
$\gamma(\infty)$. 
The union $X\cup \partial X$ has a natural topology
which restricts to the usual topology on $X$ and such that
$X$ is dense in $X\cup \partial X$. The isometry
group of $X$ acts as a group of homeomorphisms on 
$X\cup \partial X$.

A subset $C\subset X$ is 
\emph{convex} if for $x,y\in C$ the geodesic connecting
$x$ to $y$ is contained in $C$ as well. 
For every complete convex set $C\subset X$
and every $x\in X$ there is a unique point
$\pi_C(x)\in C$ of smallest distance to $x$ (Proposition II.2.4 of 
\cite{BH99}).
Now let $J\subset \mathbb{R}$ be a closed connected set and
let $\gamma:J\to X$ be a geodesic arc. Then $\gamma(J)\subset X$
is complete and convex and hence there is a 
shortest distance projection $\pi_{\gamma(J)}:X\to \gamma(J)$.
The projection $\pi_{\gamma(J)}:X\to \gamma(J)$ is 
distance non-increasing.

The following definition is due to Bestvina and
Fujiwara (Definition 3.1 of \cite{BF08}).

\begin{definition}\label{contracting}
A geodesic arc $\gamma:J\to X$ is 
\emph{$B$-contracting}
for some $B>0$ if for every closed 
metric ball $K$ in $X$ which is disjoint
from $\gamma(J)$ the diameter of the 
projection $\pi_{\gamma(J)}(K)$ does not exceed $B$. 
\end{definition}
We call a geodesic \emph{contracting} if it is $B$-contracting
for some $B>0$.
As an example, every geodesic in a ${\rm CAT}(\kappa)$-space for some
$\kappa<0$ is $B$-contracting for a number $B=B(\kappa)>0$ only
depending on $\kappa$.

The next lemma (Lemma 3.2 and 3.5 of \cite{BF08}) shows that 
a triangle containing a $B$-contracting
geodesic as one of its sides is
uniformly thin.

\begin{lemma}\label{thintriangle} 
Let $\gamma:[a,b]\to X$ be a $B$-contracting geodesic.
If $x\in X$ and if $a= \pi_{\gamma[a,b]}(x)$ then for every
$t\in [a,b]$ the geodesic connecting $x$ to $\gamma(t)$
passes through the $3B+1$-neighborhood of $\gamma(a)$. 
\end{lemma}

On the other hand, 
thinness of triangles with
a fixed geodesic $\gamma$ as one of the three sides guarantees that
$\gamma$ is contracting. This is formulated in the
following useful criterion to
detect contracting geodesics. For its formuation,
note that in a ${\rm CAT}(0)$-space, the 
\emph{Alexandrov angle} between two geodesics issuing
from the same point is defined (see \cite{BH99}).

\begin{lemma}\label{quadrangle}
Let $\gamma:J\to X$ be a geodesic such that
there is a number $B>0$ with the following property.
Assume that
for all $[a,b]\subset J$ with $\vert b-a\vert \geq B/4$
and every geodesic quadrangle $Q$ in $X$ with 
one side $\gamma[a,b]$ and an angle at least $\pi/2$ at
$\gamma(a),\gamma(b)$ the geodesic arc connecting
the two vertices of $Q$ which are distinct from
$\gamma(a),\gamma(b)$ passes through the
$B/4$-neighborhood of $\gamma[a,b]$. Then $\gamma(J)$ is
$B$-contracting. 
\end{lemma}
\begin{proof} Let $\gamma:J\to X$ be a geodesic
which satisfies the assumption in the lemma.
We have to show that 
$d(\pi_{\gamma(J)}(x),\pi_{\gamma(J)}(y))\leq B$
for all $x\in X$ with 
$d(x,\gamma(J))=R >0$ and every $y\in X$ with
$d(x,y)<R$.

For this let 
$x\in X$ with $d(x,\gamma(J))=R>0$.
Assume that $\pi_{\gamma(J)}(x)=\gamma(a)$. Let $y\in X$
and write 
$\pi_{\gamma(J)}(y)=\gamma(c)$ where $c\geq a$
without loss of generality (otherwise reverse the orientation of 
$\gamma$). The angle at $\gamma(c)$ between the
geodesic connecting $\gamma(c)$ to $y$ and the
subarc of $\gamma$ (with reversed orientation) 
connecting $\gamma(c)$ to 
$\gamma(a)$ is not smaller than $\pi/2$. Since the
angle sum of a triangle in a ${\rm CAT}(0)$-space does not
exceed $\pi$, this implies that if
$c>a+B/4$ then  
the angle at $\gamma(a+B/4)$ of the quadrangle
with vertices $x,\gamma(a),\gamma(a+B/4),y$ is 
not smaller than $\pi/2$. Thus by the assumption in the lemma,
the geodesic connecting $x$ to $y$ passes through a point $z$ 
in the $B/4$-neighborhood of $\gamma[a,a+B/4]$. 
Then $\pi_{\gamma(J)}(z)\in \gamma[a-B/4,a+B/2]$, 
moreover also $d(x,z)\geq R-B/4$. 

Now assume that $c\geq a+B$. Then
we have $d(\pi_{\gamma(J)}(z),\pi_{\gamma(J)}(y))\geq  B/2$. Since
the projection $\pi_{\gamma(J)}$ is distance
non-increasing 
we conclude that 
\[d(x,y)=
d(x,z)+d(z,y)\geq R-B/4+B/2>R.\]
In other words, $\gamma(J)$ is $B$-contracting.
\end{proof}

\section{${\rm CAT}(0)$-geometry of Weil-Petersson space}

Let $S$ be an oriented surface of genus $g\geq 0$ with $m\geq 0$
punctures and $3g-3+m\geq 2$. 
The metric completion $\overline{{\cal T}(S)}$ of the Teichm\"uller space
${\cal T}(S)$ of $S$ 
with respect to the Weil-Petersson metric 
$d_{WP}$ is a Hadamard space. The completion locus
$\overline{{\cal T}(S)}-{\cal T}(S)$ of ${\cal T}(S)$ 
can be described as follows \cite{M76}.

A \emph{surface with nodes} is defined by a degenerate
hyperbolic metric on $S$ where at least one essential simple closed
curve on $S$ (i.e. a curve which is homotopically nontrivial
and not freely homotopic into a puncture)
has been pinched to a pair of punctures. For the
free homotopy class of an essential simple closed curve $c$ on $S$, the
degenerate surfaces with a single node at $c$ define a \emph{stratum}
${\cal T}(S)_c$ in the completion locus $\overline{{\cal T}(S)}-
{\cal T}(S)$ of Teichm\"uller space. This stratum equipped with the
induced metric is isometric to the Teichm\"uller space 
equipped with the Weil-Petersson metric 
of the (possibly disconnected) surface obtained from
$S-c$ by replacing each of the two ends corresponding to $c$ 
by a cusp. If $S_c$ is disconnected then a point in the
Teichm\"uller space of $S-c$ is given by a pair of points,
one for each of the two components of $S-c$.  
The stratum ${\cal T}(S)_c$ is a convex subset of 
$\overline{{\cal T}(S)}$. The completion
locus $\overline{{\cal T}(S)}-{\cal T}(S)$ then is the
union of the completions of the strata ${\cal T}(S)_c$ where
$c$ runs through all free homotopy classes of simple closed
curves and with the obvious identifications.

The extended mapping class group ${\rm Mod}(S)$ of all isotopy
classes of diffeomorphisms of $S$ acts on $({\cal T}(S),d_{WP})$ 
properly discontinuously as a group of isometries. This action extends
to an action on the completion 
$\overline{{\cal T}(S)}$
preserving the completion
locus $\overline{{\cal T}(S)}-{\cal T}(S)$.

Masur and Wolf \cite{MW02} used the action of the 
extended mapping class group on the completion locus
of Teichm\"uller space to show.

\begin{proposition}\label{isometrygroup}
The isometry group of $({\cal T}(S),d_{WP})$ coincides with
the extended mapping class group.
\end{proposition}

For $\epsilon >0$ let ${\cal T}(S)_\epsilon$ be the
subset of ${\cal T}(S)$ of all hyperbolic metrics
whose \emph{systole}, i.e. the length of a shortest
closed geodesic, is at least $\epsilon$.
The mapping class group preserves
${\cal T}(S)_\epsilon$ and acts on it properly
discontinuously and cocompactly.
In particular, the Weil-Petersson distance between
${\cal T}(S)_\epsilon$ and the completion
locus $\overline{{\cal T}(S)}-{\cal T}(S)$ of 
Teichm\"uller space is positive. Moreover, the sectional
curvature of the restriction of the Weil-Petersson metric
to ${\cal T}(S)_\epsilon$ is bounded from above
by a negative constant.

This fact together with Lemma \ref{quadrangle} is 
used to show that Weil-Petersson
geodesic segments which are entirely contained in the thick
part of Teichm\"uller space are contracting.
An analogous result for the \emph{Teichm\"uller metric}
on Teichm\"uller space (which is however much more difficult)
was established by Minsky \cite{Mi96}.
As a convention,
in the sequel a Weil-Petersson geodesic is always
parametrized on a closed connected subset of $\mathbb{R}$.

\begin{lemma}\label{wpcontr}
For every $\epsilon >0$ there is a number $B=B(\epsilon)>0$
such that every geodesic 
$\gamma:J\to {\cal T}(S)_\epsilon$ is $B$-contracting.
\end{lemma}
\begin{proof}
It was shown in Lemma 3.1 of \cite{H08e} (see also
\cite{BF08,BMM08} for an earlier argument along the same 
line) that for every $\epsilon >0$ there is a constant
$B=B(\epsilon)>0$ only depending on $\epsilon$ such that
every geodesic $\gamma:J\to {\cal T}(S)_\epsilon$ 
satisfies the hypothesis in Lemma \ref{quadrangle} for
$B$. \end{proof}

For an isometry $g$ of ${\cal T}(S)$ define the \emph{displacement function}
$d_g$ of $g$ to be the function $x\to d_g(x)=d(x,gx)$.

\begin{definition}\label{semisimple}
An isometry $g$ of $\overline{{\cal T}(S)}$ 
is called \emph{semisimple} if $d_g$ achieves
its minimum in $\overline{{\cal T}(S)}$. 
If $g$ is semisimple and ${\rm min}\,d_g=0$
then $g$ is called \emph{elliptic}. A semisimple isometry
$g$ with ${\rm min}\,d_g>0$ is called \emph{axial}.
\end{definition}

By the above definition, an isometry is elliptic if and only
if it fixes at least one point in $\overline{{\cal T}(S)}$. 
By Proposition 3.3 of \cite{B95}, an isometry $g$ of 
$\overline{{\cal T}(S)}$ is
axial if and only if there is a geodesic
$\gamma:\mathbb{R}\to \overline{{\cal T}(S)}$ such that
$g\gamma(t)=\gamma(t+\tau)$ for every $t\in \mathbb{R}$ where
$\tau=\min\,d_g>0$.
Such a geodesic is called an \emph{oriented axis} for $g$.
Note that the geodesic $t\to \gamma(-t)$ is an oriented 
axis for $g^{-1}$. The endpoint $\gamma(\infty)$ of $\gamma$
is a fixed point for the action of $g$ on 
$\partial\overline{{\cal T}(S)}$
which is called the \emph{attracting fixed point}. 
The closed convex set $A\subset \overline{{\cal T}(S)}$ 
of all points for which the displacement
function of $g$ is minimal is 
isometric to $C\times \mathbb{R}$ where
$C\subset A$ is closed and convex (Theorem II.2.14 of
\cite{BH99}). For each
$x\in C$ the set $\{x\}\times \mathbb{R}$ is an axis of $g$.

By the Nielsen-Thurston classification,
a mapping class $g\in {\rm Mod}(S)$ either is
\emph{pseudo-Anosov} or it is of finite order or it
is \emph{reducible}. An example of a reducible mapping class
is a \emph{multi-twist} which can
be represented in the form 
$\phi_1^{k_1}\circ\dots\circ \phi_\ell^{k_\ell}$ where
each $\phi_i$ is a Dehn-twist about a simple closed curve
$c_i$ in $S$ and where the curves $c_i$ are pairwise disjoint.
We allow the multi-twist to be trivial.
We have.

\begin{lemma}\label{semisimple2}
Every isometry $\phi$ of $\overline{{\cal T}(S)}$ is semi-simple, and 
$\phi$ is elliptic if and
only if there is some $k\geq 1$ such that
$\phi^k$ is a multi-twist.
\end{lemma}
\begin{proof}
In a Hadamard space $X$, an isometry $g$ with a finite orbit 
on $X$ has a fixed point which is the center of the orbit.
This means the following.
For a fixed orbit $\{x_1,\dots,x_k\}\subset X$ for $g$ 
there is a unique point $y\in X$ such that 
the radius of the smallest 
closed metric ball centered at $y$ which contains
the set $\{x_1,\dots,x_k\}$ is minimal
(Proposition II.2.7 of \cite{BH99}).
Since this point is defined by purely metric properties,
it is a fixed point for $g$. 

As a consequence, 
an element $g\in {\rm Mod}(S)$ is elliptic if and only
if this holds true for $g^k$ for every $k>0$, and every
element of finite order is elliptic.

Now assume that
$g\in {\rm Mod}(S)$ is a multi-twist about 
a multi-curve $c=c_1\cup\dots\cup c_\ell$. Let 
$\overline{{\cal T}(S)}_c$ be the completion of the 
stratum in $\overline{{\cal T}(S)}$
of all surfaces with nodes at the curves $c_1,\dots,c_\ell$. 
Then $g$ fixes each point in $\overline{{\cal T}(S)}_c$ and
hence $g$ is elliptic.

If $g$ is pseudo-Anosov then it was shown 
in \cite{DW03} that $g$ has an 
axis in ${\cal T}(S)$ and hence it is
axial. Now assume that $g$ is reducible.
Then up to replacing $g$ by $g^k$ for some $k>0$,
$g$ preserves a non-trivial multi-curve $c$ 
component-wise, and it preserves each connected component
of $S-c$. Moreover, the multi-curve $c$ can be chosen in
such a way that for every component $S_0$ of $S-c$, either
$S_0$ is a three-holed sphere or the restriction of $g$ to $S_0$
is pseudo-Anosov.
If $g$ is not a multi-twist then 
there is at least one component $S_0$ of $S-c$ such that the restriction of
$g$ to $S_0$ is pseudo-Anosov. 
Then the restriction of $g$ to $S_0$ 
viewed as an element of the mapping class group
of $S_0$ has an axis in ${\cal T}(S_0)$.  
The Weil-Petersson metric on the stratum
${\cal T}(S)_c$ induced from the Weil-Petersson metric on $S$ 
is the product of the Weil-Petersson metrics
on the Teichm\"uller spaces of the connected components of $S-c$.
This implies that
the restriction of $g$ to ${\cal T}(S)_c$ has an axis.
Now the completion 
$\overline{{\cal T}(S)}_c\subset \overline{{\cal T}(S)}$ 
of ${\cal T}(S)_c$ is a closed convex subset of $\overline{{\cal T}(S)}$.
The shortest distance projection $\overline{{\cal T}(S)}\to
\overline{{\cal T}(S)}_c$ is distance non-increasing and
equivariant with respect to the action of $g$. 
Therefore the infimum of the displacement function $d_g$ of $g$
equals the infimum of $d_g$ on $\overline{{\cal T}(S)}_c$.
Thus this infimum is a minimum and 
once again, $g$ is axial.
\end{proof}

{\bf Remark:} 1) The flat strip theorem (Theorem II.2.14 of
\cite{BH99}) states that two geodesic lines in a Hadamard
space $X$ whose endpoints in the boundary $\partial X$ coincide
bound a flat strip. Since 
the sectional curvature of the Weil-Petersson metric
is negative, 
the proof of Lemma \ref{semisimple2} 
shows that an axial 
isometry of $\overline{{\cal T}(S)}$ which admits
an axis intersecting ${\cal T}(S)$ is pseudo-Anosov.

2) The celebrated solution of the Nielsen realization
problem states that each finite subgroup of 
${\rm Mod}(S)$ has a fixed point in ${\cal T}(S)$ \cite{K83}.
The discussion in the proof of Lemma \ref{semisimple2} 
immediately implies that such a group has a fixed point
in $\overline{{\cal T}(S)}$. It is not difficult to 
establish that there is also a fixed point in 
${\cal T}(S)$, however we omit this discussion here.

\bigskip

The following definition is due to Bestvina
and Fujiwara (Definition 5.1 of \cite{BF08}).

\begin{definition}\label{rankone}
An isometry $g$ of a ${\rm CAT}(0)$-space $X$ is called 
\emph{$B$-rank-one} for some $B>0$ 
if $g$ is axial and admits a $B$-contracting axis.
\end{definition}

We call an isometry $g$ \emph{rank-one} if $g$ is $B$-rank-one for some
$B>0$.
Since a pseudo-Anosov element has an axis $\gamma$ in ${\cal T}(S)$,
by invariance and cocompactness of the action of $g$ on 
$\gamma$, the geodesic $\gamma$ 
entirely remains in ${\cal T}(S)_\epsilon$ for some 
$\epsilon >0$. Thus the following result
(Proposition 8.1 of \cite{BF08}) is an immediate
consequence of Lemma \ref{wpcontr}.

\begin{proposition}\label{flatrank}
A pseudo-Anosov element in ${\rm Mod}(S)$ is rank-one.
\end{proposition}

{\bf Example:} An axial isometry $g$ of $\overline{{\cal T}(S)}$
which admits an axis $\gamma$ bounding a flat half-plane is
not rank-one. An example of such an axial isometry of
$\overline{{\cal T}(S)}$ can be obtained as follows. Let $c$ be a simple
closed separating curve on $S$ 
such that none of the two components 
$S_1,S_2$ of $S-c$ is a three-holed sphere. 
Let $g\in {\rm Mod}(S_1)$ be a pseudo-Anosov mapping class.
Then $g$ defines a reducible element in ${\rm Mod}(S)$.
If $\gamma_1$ is the axis for the action of $g$
on ${\cal T}(S_1)$ then for each point $z\in {\cal T}(S_2)$ the
curve $(\gamma_1,z)$ is an axis for the action of 
$g$ on the stratum ${\cal T}(S)_c\subset \overline{{\cal T}(S)}$ and hence
$(\gamma_1,z)$ is an axis for the action of $g$ on $\overline{{\cal T}(S)}$.
In particular, for every infinite geodesic $\zeta:\mathbb{R}\to
{\cal T}(S_2)$ the set $\{(\gamma(t),\zeta(s))\mid s,t\in \mathbb{R}\}  
\subset{\cal T}(S_1)\times {\cal T}(S_2)\subset 
\overline{{\cal T}(S)}$ is an isometrically embedded
euclidean plane in $\overline{{\cal T}(S)}$
containing an axis for $g$.
Thus $g$ is axial but not rank-one.

\bigskip

A homeomorphism $g$ of a topological space $K$ is said
to act with \emph{north-south dynamics} if there are
two fixed points $a\not=b\in K$ for the action of $g$ such that
for every neighborhood $U$ of $a$, $V$ of $b$ there is some 
$k>0$ such that $g^k(K-V)\subset U$ and 
$g^{-k}(K-U)\subset V$.
The point $a$ is called the \emph{attracting fixed point} for $g$,
and $b$ is the \emph{repelling fixed point}.

Teichm\"uller space equipped with the \emph{Teichm\"uller metric}
can be compactified by adding the \emph{Thurston boundary}  
${\cal P\cal M\cal L}$ 
of projective measured geodesic laminations
which is a topological sphere. This compactification however is
different from the ${\rm CAT}(0)$-boundary 
$\partial\overline{{\cal T}(S)}$ 
of $\overline{{\cal T}(S)}$.  
The action of the extended mapping class group on ${\cal T}(S)$ naturally 
extends to an action on ${\cal P\cal M\cal L}$.
An element $g\in {\rm Mod}(S)$ acts on 
${\cal P\cal M\cal L}$ with north-south-dynamics if and
only if $g$ is pseudo-Anosov. 
Lemma 3.3.3 of \cite{B95} shows that a rank-one isometry
of a proper Hadamard space $X$ acts on the boundary 
$\partial X$ with north-south dynamics. The proof of this fact
given in \cite{H08d} (proof of Lemma 4.4) does not
use the assumption of properness of $X$. Thus we obtain.

\begin{lemma}\label{northsouth}
A rank-one isometry $g$ of $\overline{{\cal T}(S)}$ 
acts with north-south dynamics on $\partial\overline{{\cal T}(S)}$.
\end{lemma}

\section{Non-elementary groups of isometries}

In this section we investigate the action 
on $\partial\overline{{\cal T}(S)}$ of non-elementary 
subgroups of ${\rm Mod}(S)$ which contain a pseudo-Anosov element.
We begin with recalling some standard terminology used
for groups of isometries on Hadamard spaces.

Let $G<{\rm Mod}(S)$ be any subgroup.
The \emph{limit set} $\Lambda$ of $G$ is the set
of accumulation points in $\partial \overline{{\cal T}(S)}$ 
of one (and hence every)
orbit of the action of $G$ on $\overline{{\cal T}(S)}$. 
If $g\in G$ is axial with axis $\gamma$, then $\gamma(\infty),
\gamma(-\infty)\in \Lambda$. In other words, the two fixed
points for the action 
of a pseudo-Anosov element on 
$\partial \overline{{\cal T}(S)}$ are contained in $\Lambda$.

\begin{lemma}\label{limitsetmod}
The limit set of ${\rm Mod}(S)$ is the entire boundary
$\partial \overline{{\cal T}(S)}$ of 
$\overline{{\cal T}(S)}$.
\end{lemma}
\begin{proof} For suficiently
small $\epsilon >0$, the set ${\cal T}(S)_\epsilon\subset
{\cal T}(S)$ of all hyperbolic
metrics whose systole is at least $\epsilon$ is connected,
and the mapping class group ${\rm Mod}(S)$  
acts cocompactly
on ${\cal T}(S)_\epsilon$.
There is a number $R_0>0$ such that the
Weil-Petersson distance between
any point in $\overline{{\cal T}(S)}$ and 
${\cal T}(S)_\epsilon$ is at most $R_0$ \cite{W03}.
Thus there is a number
$R_1>R_0$ such that 
for all $x\in {\cal T}(S)_\epsilon$ and all $y\in\overline{{\cal T}(S)}$ 
there is some 
$g\in {\rm Mod}(S)$ with $d(gx,y)\leq R_1$. 
This just means
that $\partial \overline{{\cal T}(S)}$ is the limit set of ${\rm Mod}(S)$.
\end{proof}

\begin{lemma}\label{minimal}
Let $G<{\rm Mod}(S)$ be a subgroup 
which contains a pseudo-Anosov element $g$.
Then the limit set $\Lambda$ of 
$G$ is the closure in $\partial \overline{{\cal T}(S)}$ of the set of 
fixed points of conjugates of $g$ in $G$.  
If $G$ is non-elementary then 
$\Lambda$ does not have isolated points.
\end{lemma}
\begin{proof}
Let $G<{\rm Mod}(S)$ be a subgroup which 
contains a pseudo-Anosov element $g\in G$.
Let $\Lambda$ be the limit set of $G$.
We claim that $\Lambda$ is contained in the
closure of the $G$-orbit of the two fixed 
points of $g$. For this let $\gamma$ be the axis
of $g$. By Proposition \ref{flatrank}, $\gamma$
is $B$-contracting for some $B>0$. Let $\xi\in \Lambda$ and let  
$(g_i)\subset G$ 
be a sequence such that $(g_i\gamma(0))$ converges to 
$\xi$. There are two cases possible.

In the first case, up to passing to a 
subsequence, the geodesics $g_i\gamma$ eventually leave
every bounded set. Let $x_0=\gamma(0)$ and for $i\geq 1$ let 
$x_i=\pi_{g_i\gamma(\mathbb{R})}(\gamma(0))$. 
Then $d(x_0,x_i)\to \infty$ $(i\to \infty)$. 
On the other hand, $g_i\gamma$ is $B$-contracting and hence 
by Lemma \ref{thintriangle} 
a geodesic $\zeta_i$ connecting $x_0$ to $g_ix_0$ passes through
the $3B+1$-neighborhood of $x_i$, and the same
is true for a geodesic $\eta_i$ connecting 
$x_0$ to $g_i\gamma(\infty)$.
By ${\rm CAT}(0)$-comparison, the angles at $x_0$
between the geodesics $\zeta_i,\eta_i$ 
converge to zero as $i\to \infty$. Since 
$g_ix_0\to \xi$, 
the sequence $(g_i\gamma(\infty))$ 
converges to $\xi$ as well.
But $g_i\gamma(\infty)$
is a fixed point of the conjugate $g_igg_i^{-1}$ of $g$. 
Thus $\xi$ is contained in the closure of the
fixed points of all conjugates of $g$.

In the second case there is a bounded neighborhood $K$ of $x_0$ in  
$\overline{{\cal T}(S)}$
such that $g_i\gamma\cap K\not=\emptyset$ for all $i$.
For $i>0$ let $\zeta_i$ be the geodesic connecting 
$x_0$ to $g_ix_0$. 
Since $g_ix_0\to \xi$, 
the geodesics $\zeta_i$ converge as 
$i\to \infty$ locally uniformly to the geodesic ray connecting $x_0$ to $\xi$.

After passing to a subsequence and perhaps a change
of orientation of $\gamma$ we may assume that
for large $i$ the point $g_ix_0$ lies between a point
$z_i\in g_i\gamma\cap K$ and $g_i\gamma(\infty)$. This means that
$g_ix_0$ is contained in the geodesic connecting $z_i$ to 
$g_i\gamma(\infty)$. Since the distance between 
$z_i$ and $x_0$ is uniformly bounded, 
by ${\rm CAT}(0)$-comparison the 
Alexandrov angle at
$g_ix_0$ between the inverse of the
geodesic $\zeta_i$ (which connects $g_ix_0$ to $x_0$) and
the inverse of the geodesic $g_i\gamma$ (which connects
$g_ix_0$ to $z_i$) tends to zero as $i\to \infty$.
This implies that the angle at $g_ix_0$ of 
the ideal triangle in $\overline{{\cal T}(S)}$ with
vertices $x_0,g_ix_0,g_i\gamma(\infty)$ tends to $\pi$ as
$i\to \infty$.

Since in a ${\rm CAT}(0)$-space the sum of the Alexandrov
angles of a geodesic triangle (with possibly one
vertex at infinity) does not exceed $\pi$, 
the angle at $x_0$ between the geodesic 
$\zeta_i$ and the geodesic $\rho_i$ connecting
$x_0$ to $g_i\gamma(\infty)$ tends to zero as $i\to \infty$.
But $g_ix_0\to \xi$ and therefore
the points $g_i\gamma(\infty)$ converge to $\xi$
$(i\to \infty)$ in $\partial\overline{{\cal T}(S)}$.
Thus $\xi$ is indeed contained in the closure of the 
fixed points of conjugates of $g$.  

Now assume that 
the limit set $\Lambda$ of $G$ contains at least 3 points. 
Let $g$ be any pseudo-Anosov element of $G$. 
Since by Lemma \ref{northsouth}
$g$ acts with north-south dynamics on 
$\partial \overline{{\cal T}(S)}$, 
the set $\Lambda$ 
contains at least one point $\xi$ which is not a fixed point
of $g$. The sequence $(g^k\xi)$ consists of pairwise
distinct points which converge as $k\to\infty$ 
to the attracting fixed point of $g$. Similarly,
the sequence $(g^{-k}\xi)$ consists of pairwise distinct points which
converge as $k\to\infty$ to the repelling fixed point of $g$. 
Moreover, by the above, a point $\xi\in \Lambda$ which is
not a fixed point of a pseudo-Anosov element of $G$ is a limit
of fixed points of pseudo-Anosov elements. This shows that
$\Lambda$ does not have
isolated points and completes the proof of the lemma.
\end{proof}

We need the following simple (and well known to the experts)
observation which parallels
the properties of the action of 
${\rm Mod}(S)$ on the space of projective measured
geodesic laminations. This observation follows
immediately from the work of Brock, Masur and Minsky
\cite{BMM08}.

\begin{lemma}\label{pseudofix}
Let $g,h\in {\rm Mod}(S)$ be pseudo-Anosov elements.
If there is a common fixed point for
the action of $g,h$ on $\partial\overline{{\cal T}(S)}$ then
the fixed point sets of $g,h$ coincide.
\end{lemma}
\begin{proof}
Let $g,h\in {\rm Mod}(S)$ be pseudo-Anosov elements
and assume that there is a common fixed point
for the action of $g,h$ on $\partial \overline{{\cal T}(S)}$.
Since $g,h$ act with
north-south dynamics on $\partial\overline{{\cal T}(S)}$,
this implies that the axis $\gamma$ for $g$ and
the axis $\eta$ for $h$ have a common endpoint, 
say $\gamma(\infty)=\eta(\infty)$. We may also
assume that $\gamma(\infty)$ is the attracting fixed
point for both $g,h$.

By Theorem 1.5 of \cite{BMM08}, up to a reparametrization
we have $d_{WP}(\gamma(t),\eta(t))\to 0$
$(t\to \infty)$. 
After another reparametrization, there
is a number $r>0$ such that the semi-group 
$\{g^k\mid k\geq 0\}$ acts cocompactly on 
the closed $2r$-neighborhood $N\subset {\cal T}(S)$ of 
$\gamma[0,\infty)$, and that the $r$-neighborhood of 
$\gamma[0,\infty)$ contains $\eta[0,\infty)$. 

Let $\tau_0$ be the translation length of 
$g$ on $\gamma$.
If there are no integers $k,\ell >0$ such that $g^k=h^\ell$
then there are infinitely many 
distinct elements of ${\rm Mod}(S)$
of the form $g^{-m}h^{n}$ 
which map $\eta(0)$ into the $r$-neighborhood 
of $\gamma[0,\tau_0]$. Namely, let $n>0$ be arbitrary.
Then there is a unique number $m\in \mathbb{Z}$ such that
$\pi_{\gamma(\mathbb{R})}h^n(\eta(0))\in 
\gamma[m\tau_0,(m+1)\tau_0)$. Since the point $h^n(\eta(0))$ is contained
in the $r$-neighborhood of $\gamma(\mathbb{R})$, the point 
$g^{-m}h^n(\eta(0))$ is contained in the $r$-neighborhood of
$\gamma[0,\tau_0]$. However, this 
violates the fact that ${\rm Mod}(S)$ acts 
properly discontinuously on ${\cal T}(S)$.
Thus there are number $k,\ell>0$ with $g^k=h^\ell$ and hence
the fixed point sets for the action of $g,h$ on 
$\partial\overline{{\cal T}(S)}$ coincide.
\end{proof}

The action of a group $G$ on a topological space $Y$ is
called \emph{minimal} if every $G$-orbit is dense.

\begin{lemma}\label{moveaway}
Let $G<{\rm Mod}(S)$ be a non-elementary group 
with limit set $\Lambda$ 
which contains
a pseudo-Anosov element $g\in G$ with fixed points $a\not=b\in \Lambda$.
Then for every non-empty open set $V\subset \Lambda$ 
there is some $u\in G$ with $u\{a,b\}\subset V$. 
Moreover, the action of $G$ on 
$\Lambda$ is minimal.
\end{lemma}
\begin{proof}
Let $G<{\rm Mod}(S)$ be a non-elementary subgroup
with limit set $\Lambda$ which contains 
a pseudo-Anosov element $g\in G$.
Let $a,b\in \Lambda$ be the attracting and 
repelling fixed points of $g$, respectively, and 
let $V\subset \Lambda$ be a non-empty open set.
By Lemma \ref{minimal}, the limit set $\Lambda$ does not
have isolated points and up to replacing 
$g$ by $g^{-1}$ (and exchanging $a$ and $b$) there is an element
$v\in G$ which maps $a$ to $v(a)\in V-\{a,b\}$. 
Then $h=vgv^{-1}$ is a
pseudo-Anosov element  
with fixed points $v(a)\in V-\{a,b\},v(b)
\in \Lambda$. By Lemma \ref{pseudofix}, 
we have $v(b)\not\in \{a,b\}$. 
By Lemma \ref{northsouth}, 
$h$ acts with north-south 
dynamics on $\partial\overline{{\cal T}(S)}$ and hence
$h^k\{a,b\}\subset V$ 
for all sufficiently large $k$.

Every closed $G$-invariant subset $A$ of 
$\partial \overline{{\cal T}(S)}$
contains every fixed point of every pseudo-Anosov element.
Namely, if $a\not= b$ are the two fixed points of a
pseudo-Anosov element $g\in G$ and if there is some $\xi\in A-\{a,b\}$
then also $\{a,b\}\subset A$ since $A$ is closed
and $g$ acts with north-south dynamics on 
$\partial\overline{{\cal T}(S)}$.
On the 
other hand, if $a\in A$ then by the above consideration
there is some $h\in G$ with $h(a)\in \Lambda-\{a,b\}$
and once again, we conclude by invariance that $b\in A$ as well.
Now the set of all fixed points of 
pseudo-Anosov elements of $G$ is $G$-invariant and hence the smallest
non-empty closed $G$-invariant subset of 
$\partial \overline{{\cal T}(S)}$ is the
closure of the set of fixed points of pseudo-Anosov elements.
This set contains
the limit set $\Lambda$ of $G$ by Lemma \ref{minimal} and hence
it coincides with $\Lambda$. In other words, 
the action of $G$ on $\Lambda$ is minimal.
The lemma is proven.
\end{proof}

Note the following immediate corollary of 
Lemma \ref{pseudofix}.

\begin{corollary}\label{nofixinf}
Let $G<{\rm Mod}(S)$ be a non-elementary subgroup
which contains a pseudo-Anosov element.
Then $G$ does not fix a point in $\partial\overline{{\cal T}(S)}$.
\end{corollary}

{\bf Example:} There are non-elementary groups $G<{\rm Mod}(S)$
which fix a point in $\partial\overline{{\cal T}(S)}$.
Namely, let $c$ be a simple closed separating curve on $S$
so that $S-c=S_1\cup S_2$ where neither $S_1$ nor $S_2$ is 
a three-holed sphere. Let $\gamma_i\in {\rm Mod}(S_i)$ be a
pseudo-Anosov element $(i=1,2)$. Then $\gamma_1,\gamma_2$ generate
a free abelian subgroup $G$ of ${\rm Mod}(S)$ whose limit set
is a circle which is fixed pointwise by $G$.

\bigskip

We are now ready to show.

\begin{proposition}\label{pairfixdense}
Let $G<{\rm Mod}(S)$ be a non-elementary subgroup
with limit set $\Lambda$ 
which contains a pseudo-Anosov element. 
\begin{enumerate}
\item
The pairs of fixed points of pseudo-Anosov elements
of $G$ are dense in $\Lambda\times \Lambda$.
\item For any two non-empty open subsets $W_1,W_2$ of 
$\Lambda\times \Lambda$ there is some 
$g\in G$ with $gW_1\cap W_2\not=\emptyset$.
\end{enumerate}
\end{proposition}
\begin{proof}
Let $G<{\rm Mod}(S)$ be a non-elementary
subgroup with limit set $\Lambda$. Assume that $G$  
contains a pseudo-Anosov element $g$ with
attracting fixed point $a\in \Lambda$ and repelling fixed point
$b\in \Lambda$.

Let $U\subset \Lambda\times \Lambda$ 
be a non-empty open set. Our goal is to show 
that $U$ contains a pair of fixed points
of a pseudo-Anosov element. Since $\Lambda$ does not
contain isolated points, for this
we may assume that there are small open sets 
$V_i\subset \partial\overline{{\cal T}(S)}-\{a,b\}$ 
with disjoint closure $\overline{V_i}$ $(i=1,2)$ and
such that $U=V_1\times V_2\cap \Lambda\times \Lambda$.

Choose some $u\in G$ which maps $\{a,b\}$ into $V_1$.
Such an element exists by Lemma \ref{moveaway}.
Then $v=ugu^{-1}$ is a pseudo-Anosov element
with fixed points $ua,ub\in V_1$.
Similarly, there is a pseudo-Anosov element  
$w\in G$ with both fixed points in $V_2$. Via replacing $v,w$ by
sufficiently high powers we may assume that
$v(\partial\overline{{\cal T}(S)}-V_1)
\subset V_1,v^{-1}(\partial \overline{{\cal T}(S)}-V_1)\subset V_1$ 
and that $w(\partial\overline{{\cal T}(S)}-V_2)\subset V_2,
w^{-1}(\partial\overline{{\cal T}(S)}-V_2)\subset V_2$.
Then we have $wv(\partial \overline{{\cal T}(S)}-V_1)\subset V_2$ 
and $v^{-1}w^{-1}(\partial\overline{{\cal T}(S)} -V_2)\subset V_1$.
By a result of McCarthy \cite{MC85}, up to possibly replacing
$v$ and $w$ by even higher powers 
we may assume that $wv$ is pseudo-Anosov.
Then $wv$ acts on $\partial\overline{{\cal T}(S)}$ with
north-south-dynamics. Since $wv(\overline{V_2})\subset V_2$ and
$v^{-1}w^{-1}(\overline{V_1})\subset V_1$, 
the pair of fixed points of $wv$ is necessarily contained in
$V_1\times V_2$ and hence in $U$.
The first part of the proposition is proven.

To show the second part of the proposition, let  
$W_1,W_2\subset\Lambda\times \Lambda$ be non-empty open sets.
We have to show that there is some $g\in G$ such that
$gW_1\cap W_2\not=\emptyset$. For this
we may assume without loss of generality that
$W_1=U_1\times U_2,W_2=U_3\times U_4$ where $U_1,U_2$
and $U_3,U_4$  
are non-empty open subsets of $\Lambda$ with disjoint 
closure. Since $\Lambda$ does not have isolated points, 
by possibly 
replacing $U_i$ by proper non-empty open subsets we may assume that
the sets $U_i$ are pairwise disjoint.

By the first part of the proposition, there is a pseudo-Anosov element
$u\in G$ with attracting fixed point in $U_1$ and 
repelling fixed point in $U_4$.
Since $u$ acts on $\partial \overline{{\cal T}(S)}$ with north-south dynamics,
there is some $k>0$ and a small open neighborhood 
$U_5\subset U_1$ of the attracting fixed point of $u$ 
such that $u^{-k}(U_5\times U_2)\subset
U_1\times U_4$. The same argument produces an element
$w\in G$, a number $\ell >0$ and an open subset 
$U_6$ of $U_2$ such that $w^\ell(u^{-k}(U_5\times U_6))\subset
U_3\times U_4$. This completes the proof of the proposition.
\end{proof}

As noted in the example after Corollary \ref{nofixinf},
in general the second part of 
Proposition \ref{pairfixdense}
does not hold true for 
non-elementary subgroups of ${\rm Mod}(S)$
which do not contain a pseudo-Anosov element.

\section{The Weil-Petersson geodesic flow}

In this section we discuss some implications of the 
results in the previous section to the dynamics of the
Weil-Petersson geodesic flow on moduli space.

Let $T^1{\cal T}(S)$ be the unit tangent bundle of 
${\cal T}(S)$ for the Weil-Petersson metric.
The \emph{Weil-Petersson geodesic
flow} $\Phi^t$ acts on $T^1{\cal T}(S)$
by associating to a direction $v$ and a number
$t>0$ the unit tangent $\Phi^tv$ at $t$ of the geodesic with
initial velocity $v$. Note that this flow is not everywhere
defined due to the existence of finite length geodesics 
which end in a point in $\overline{{\cal T}(S)}-{\cal T}(S)$. 
Define
${\cal G}\subset T^1{\cal T}(S)$ to be the space of all
directions of biinfinite geodesics, i.e. such that
the flow line of $\Phi^t$ through a point $v\in {\cal G}$
is defined for all times. 
Note that the set ${\cal G}$ is 
invariant under the
action of the extended mapping class group.

The following result is due to Wolpert \cite{W03,W06}.

\begin{lemma}\label{raydense}
${\cal G}$ is a dense $G_\delta$-subset of $T^1{\cal T}(S)$ of 
full Lebesgue measure.
\end{lemma}
\begin{proof}
A direction at a point $x\in {\cal T}(S)$ either 
defines a geodesic ray (i.e. a 
geodesic
defined on the half-line $[0,\infty)$)
or a geodesic which 
ends at a point in $\overline{{\cal T}(S)}-{\cal T}(S)$.
The set $\overline{{\cal T}(S)}-{\cal T}(S)$ 
is a countable union of closed convex
strata, each of real codimension two. Since by the 
${\rm Cat}(0)$-property,
any two points $x,y\in \overline{{\cal T}(S)}$ can be connected
by a unique geodesic depending continuously on $x,y$, 
the set of directions of geodesics issuing from a point in 
${\cal T}(S)$ and which terminate in the closure of 
a fixed stratum is a closed subset of $T^1{\cal T}(S)$ of 
real codimension one. Thus ${\cal G}$ is the complement
in $T^1{\cal T}(S)$ of a countable 
union of closed subsets of codimension one, i.e. it is
a dense $G_\delta$-set (we refer to \cite{W03,W06} for details). 

Wolpert \cite{W03,W06}
also observed that for every $x\in {\cal T}(S)$
the set of directions
of geodesic rays issuing from $x$ has full Lebesgue measure in the
unit sphere at $x$. Then ${\cal G}$ has full Lebesgue measure.
\end{proof}

To each $v\in {\cal G}$ we can associate the ordered pair 
$\pi(v)\in \partial\overline{{\cal T}(S)}\times
\partial\overline{{\cal T}(S)}$ 
of endpoints of the biinfinite geodesic $\gamma$ with initial velocity
$v$ (here ordered means that $\pi(v)=(\gamma(\infty),\gamma(-\infty))$).
The map $\pi$ clearly is invariant under the action of the geodesic
flow on $T^1{\cal T}(S)$ and hence it factors through a map
of the quotient space ${\cal G}/\Phi^t$. Since the Weil-Petersson metric
is negatively curved, by the flat strip theorem (Theorem II.2.13
of \cite{BH99}) the induced map 
${\cal G}/\Phi^t\to \pi({\cal G})\subset
\partial \overline{{\cal T}(S)}\times 
\partial\overline{{\cal T}(S)}$
is injective. This means that the set $\pi({\cal G})$ can be
equipped with two natural topologies: the
topology as a quotient of ${\cal G}$, and 
the induced topology as a subset of 
$\partial \overline{{\cal T}(S)}\times \partial\overline{{\cal T}(S)}$.
We next observe that these two topologies coincide.

\begin{lemma}\label{fiber}
The map $\pi$ factors through
a ${\rm Mod}(S)$-equivariant homeomorphism of 
${\cal G}/\Phi^t$ equipped with the quotient topology
onto $\pi({\cal G})$ equipped with the topology as a subset of
$\partial \overline{{\cal T}(S)}\times 
\partial\overline{{\cal T}(S)}$.
\end{lemma} 
\begin{proof}
By the definition of the topology of $\partial\overline{{\cal T}(S)}$,
the map $\pi$ is continuous. Moreover, it is
clearly equivariant under the action of ${\rm Mod}(S)$.

We have to 
show that $\pi$ is open for the topology 
of ${\cal G}$ as a subset of $T^1{\cal T}(S)$ and for 
the topology of $\pi{\cal G}$ as a subset of 
$\partial\overline{{\cal T}(S)}\times
\partial\overline{{\cal T}(S)}$.
For this let $v\in {\cal G}$ and let
$U$ be a neighborhood of $v$ in ${\cal G}$. We have to find
a neighborhood $V$ of $\pi(v)$ in 
$\partial\overline{{\cal T}(S)}\times 
\partial \overline{{\cal T}(S)}$ 
such that the unit tangent line
of every biinfinite geodesic $\gamma$ whose pair of endpoints
is contained in $V$ passes through $U$.

For this let $x\in {\cal T}(S)$ be the footpoint of $v$.
Since a geodesic depends smoothly on its initial velocity,
if no such neighborhood $V$ of $\pi(v)$ exists 
then there is 
a sequence of points $(a_i,b_i)\subset 
\partial\overline{{\cal T}(S)}\times \partial\overline{{\cal T}(S)}$ 
which can be connected
by a geodesic line $\gamma_i$ in $\overline{{\cal T}(S)}$ and 
there is a number $\epsilon >0$ such that
$d_{WP}(x,\gamma_i(\mathbb{R}))\geq \epsilon$ for all $i$.

Let $\Delta_i$ be the ideal triangle in ${\cal T}(S)$ 
with vertices $x,a_i,b_i$. Since $a_i\to \gamma(\infty),
b_i\to \gamma(-\infty)$, the angle at $x$ of the triangle
$\Delta_i$ converges to $\pi$ as $i\to \infty$. 
Connect each point on the
geodesic ray from $x$ to $b_i$ to the point $a_i\in 
\partial\overline{{\cal T}(S)}$ by a geodesic ray.
This defines a \emph{ruled surface} in ${\cal T}(S)$ 
with smooth interior which we denote again by $\Delta_i$. 
The intrinsic Gau\ss{} curvature 
of this surface at a point $y$ 
is bounded from above by an upper bound for the curvature
of the Weil-Petersson metric at $y$. Since the Weil-Petersson 
metric is negatively curved, there is a number
$r<\epsilon$ such that the Gau\ss{} curvature of
the intersection of $\Delta_i$ with 
the ball of radius $r$ about $x$
is bounded from above by $-r$. (Such a argument has been used
in the literature many times. We refer to \cite{BMM08} for
a more detailed explanation and for additional references).

Let $\zeta_i:[0,\infty)\to {\cal T}(S)$ 
be the side of $\Delta_i$ connecting 
$x$ to $a_i$. 
The intrinsic angle at $x$ of the triangle $\Delta_i$ 
coincides with the Weil-Petersson angle at $x$.
By assumption, the distance
between $x$ and $\gamma_i$ is at least $\epsilon>r$.
If $i>0$ is sufficiently large that the angle of
$\Delta_i$ at $x$ exceeds $\pi/2$ then 
by convexity of the distance function, 
a geodesic in $\Delta_i$ for the intrinsic
metric which issues from a point in
$\zeta_i[0,r/2]$ and is perpendicular to
$\zeta_i$ does not intersect the side of 
$\Delta_i$ connecting $x$ to $b_i$. Therefore the maximal
length of such a geodesic is not smaller than
$r/2$. The union of these geodesic segments 
is an embedded strip in  
$\Delta_i$ which is 
contained in the ball of radius
$r$ about $x$ in ${\cal T}(S)$. Hence
the Gau\ss{} curvature of $\Delta_i$ at each point in
the strip is at most $-r$. Moreover, 
comparison with the euclidean plane shows that
the area of the strip is at least $r^2/4$.
Since the Gau\ss{} curvature of $\Delta_i$ is negative,
this implies that the integral of 
the Gau\ss{} curvature over $\Delta_i$ does not exceed 
$-r^3/4$.

On the other hand, the angle of $\Delta_i$ at $x$ 
tends to $\pi$ as $i\to \infty$.  
Since the Gau\ss{} curvature of $\Delta_i$ is negative, 
the Gau\ss{}-Bonnet theorem shows that 
the integral of the Gau\ss{} curvature of $\Delta_i$ tends
to zero as $i\to \infty$ which is a contradiction to the
estimate in the previous paragraph.
As a consequence,
the image of the open set $U$ under the projection
$\pi$ contains indeed an open subset
of $\pi({\cal G})$ equipped with the topology 
as a subspace of $\partial\overline{{\cal T}(S)}
\times \partial\overline{{\cal T}(S)}$ which shows the lemma.
\end{proof}

Let $T^1{\cal M}(S)$ be the quotient
of the unit tangent bundle $T^1{\cal T}(S)$ of ${\cal T}(S)$ under the
action of the (extended) mapping class group.
By equivariance, the Weil-Petersson geodesic flow projects
to a flow on $T^1{\cal M}(S)$.

Every pseudo-Anosov element
$g\in {\rm Mod}(S)$ defines a periodic orbit for the Weil-Petersson flow.
This periodic orbit is the projection of the unit
tangent line of an axis of $g$. These
are the only periodic orbits.
Namely, if $\{\Phi^tv\mid t\}$ is a periodic orbit in $T^1{\cal M}(S)$ 
then there is a biinfinite Weil-Petersson geodesic $\gamma$
in ${\cal T}(S)$ whose unit tangent line projects
to the orbit. This geodesic is invariant under an element
$g\in {\rm Mod}(S)$. Then $g$ is axial, with axis $\gamma\subset {\cal T}(S)$, 
and hence $g$ is pseudo-Anosov.

Proposition \ref{pairfixdense},
applied to the full mapping
class group, shows together with
Lemma \ref{fiber}
immediately the following result of 
Brock, Masur and Minsky \cite{BMM08}.

\begin{proposition}\label{density}
Periodic orbits are dense in $T^1{\cal M}(S)$.
\end{proposition}
\begin{proof} Since the subset ${\cal G}$ of
$T^1{\cal T}(S)$ is dense and ${\rm Mod}(S)$-invariant, 
it suffices to show that the unit tangents of all axes
of all pseudo-Anosov elements are dense in ${\cal G}$.
Now by Lemma \ref{limitsetmod} 
and Proposition \ref{pairfixdense} the set of pairs of endpoints of
all axes of pseudo-Anosov elements in ${\rm Mod}(S)$ is 
dense in $\partial \overline{{\cal T}(S)}
\times \partial\overline{{\cal T}(S)}$ and hence
in $\pi{\cal G}$ and therefore
the proposition follows from Lemma \ref{fiber}.
\end{proof}

We now use local compactness of Teichm\"uller space to complete
the proof of Theorem 1 from the introduction.

\begin{proposition}\label{thm1finish}
Let $G<{\rm Mod}(S)$ be a non-elementary subgroup 
with limit set $\Lambda$
which contains a pseudo-Anosov element.
Then there is a dense orbit
for the action of $G$ on $\Lambda\times \Lambda$.
\end{proposition}
\begin{proof} 
Let $G<{\rm Mod}(S)$ be a non-elementary subgroup which
contains a pseudo-Anosov element. Let as before 
${\cal G}\subset T^1{\cal T}(S)$ be the space of all 
directions of biinfinite Weil-Petersson geodesics
in ${\cal T}(S)$ and let ${\cal G}_0\subset {\cal G}$ be
the space of all directions of geodesics with  
both endpoints in $\Lambda$.
Then ${\cal G}_0$ is a closed $G$-invariant
subset of the (non-locally compact) space 
${\cal G}$. By Lemma \ref{fiber}, the restriction 
to ${\cal G}_0$ of the map $\pi$
factors through a homeomorphism 
${\cal G}_0/\Phi^t\to \pi({\cal G}_0)\subset
\Lambda\times \Lambda$. Since a pair of fixed 
points of a pseudo-Anosov element $g\in G$ is 
contained in $\pi({\cal G})$, by Proposition
\ref{pairfixdense} the set $\pi{\cal G}_0$ is dense
in $\Lambda\times \Lambda$.

Let $P:T^1{\cal T}(S)\to T^1{\cal T}(S)/G=N$ be the
canonical projection.
We claim that for all nonempty open sets
$U,V\subset N$ with 
$U\cap P{\cal G}_0\not=\emptyset$ and 
$V\cap P{\cal G}_0\not=\emptyset$ 
and every $t>0$ there is some
$u\in U\cap P{\cal G}_0$ and some $T>t$ such that $\Phi^Tu\in V$.
For this let $\tilde U,\tilde V$ be the preimages of 
$U,V$ in $T^1{\cal T}(S)$. Then $\tilde U,\tilde V$ are open $G$-invariant
subsets of $T^1{\cal T}(S)$.
By Lemma \ref{fiber} the projection 
$\pi:{\cal G}\to \pi{\cal G}\subset
\partial\overline{{\cal T}(S)}\times
\partial\overline{{\cal T}(S)}$ is open,
there are open subsets $W_1,W_2$ of 
$\partial\overline{{\cal T}(S)}\times \partial\overline{{\cal T}(S)}$
with $W_1\cap \pi{\cal G}\subset \pi(\tilde U)$,
$W_2\cap \pi{\cal G}\subset \pi(\tilde V)$ and such that 
the intersections of $W_1,W_2$ with $\Lambda\times \Lambda$ 
are non-empty.

Since $W_i\cap \Lambda\times \Lambda\not=\emptyset$,
by the second part of Proposition \ref{pairfixdense} 
there is some $h\in G$
such that $W=W_1\cap h^{-1}W_2\not=\emptyset$ and that
$W\cap \Lambda\times \Lambda\not=\emptyset$.
Let $(a,b)\in W\cap \pi{\cal G}_0$ be the pair of endpoints of an axis
of a pseudo-Anosov element $g\in G$. Such 
an element exists by the first part of 
Proposition \ref{pairfixdense}.
Then the axis of the conjugate $hgh^{-1}$ of $g$ in $G$ 
has a pair of endpoints $(ha,hb)$ in 
$W_2\cap \pi{\cal G}_0$. Since the unit tangent lines of axes of
pseudo-Anosov elements which are conjugate in $G$ 
project to the same 
periodic orbit in $N$ for the Weil-Petersson
flow, this implies that the projection of the
unit tangent line of the axis of $g$ passes through
both $U$ and $V$. In particular, for $t>0$ and for
a point $x\in U\cap P{\cal G}_0$ contained in this periodic orbit, there
is some $T>t$ such that $\Phi^Tx\in V$. This shows our claim.

We use this observation 
to complete the proof of the corollary.
Namely, the closure $\overline{P{\cal G_0}}$ of $P{\cal G}_0$
in $N$ is locally compact and separable since this is the case for
$N$. Moreover, $\overline{P{\cal G}_0}$ is invariant under the
Weil-Petersson flow $\Phi^t$.
Hence we can choose a countable basis $U_i$ of 
open sets for $\overline{P{\cal G}_0}$. Let $V_1=U_1$ and 
for each $i\geq 2$ define inductively
a nonempty open set $V_i$ in $\overline{P{\cal G}_0}$ with 
$\overline{V_i}\subset V_{i-1}\subset U_1$ and a number
$t_i>t_{i-1}$ such that $\Phi^{t_i}\overline{V_i}\subset U_i$. This is possible
by the above consideration and by continuity of the Weil-Petersson flow.
Then $\cap_i \overline{V_i}\not=\emptyset$, and the forward
$\Phi^t$-orbit
of any point in $\cap_i\overline{V_i}$ is infinite and 
dense in $\overline{P{\cal G}_0}$. With the same argument
we can also guarantee that the backward $\Phi^t$-orbit
of a point $v\in \cap_i \overline{V_i}$ is infinite and dense in
$\overline{P{\cal G}_0}$.  
But this just means that for a lift $\tilde v$ of $v$
the $G$-orbit of $\pi(v)\in \Lambda\times\Lambda$
is dense. This completes the proof of the corollary.
\end{proof}

A flow $\Phi^t$ is called \emph{topologically transitive}
if it admits a dense orbit. 
As an immediate consequence of Lemma \ref{fiber} and
Proposition \ref{thm1finish} we
obtain the following result of Brock, Masur and Minsky \cite{BMM08}.

\begin{corollary}
The Weil-Petersson geodesic flow on $T^1{\cal M}(S)$ is topologically 
transitive.
\end{corollary}

{\bf Remark:} For two points $\xi\not=\eta\in \partial\overline{{\cal T}(S)}$
it is in general difficult to decide whether 
$(\xi,\eta)\in \pi{\cal G}$, i.e. whether there is
a geodesic line connecting $\xi$ to $\eta$. 
However, Brock, Masur
and Minsky \cite{BMM08} showed the following.
Let $\gamma:[0,\infty)\to {\cal T}(S)$ be a geodesic
ray such that there is a number $\epsilon >0$ and
a sequence $t_i\to \infty$ with 
$\gamma(t_i)\in {\cal T}(S)_\epsilon$. Then 
$\gamma(\infty)$ can be connected
to every
$\xi\in \partial\overline{{\cal T}(S)}-\{\gamma(\infty)\}$
by a geodesic.
\bigskip

Even though the space $T^1{\cal M}(S)$ is non-compact and
the Weil-Petersson geodesic 
flow $\Phi^t$ on $T^1{\cal M}(S)$ is not everywhere defined, it admits
uncountably many invariant Borel-probability
measures. Indeed, it was shown in \cite{H08e} that there
is a continuous injection from the space of
invariant probability measures for the Teichm\"uller flow
into the space of invariant probability measures for
the Weil-Petersson geodesic flow. Specific such measures are
measures which are supported on periodic orbits.
Each such measure is \emph{ergodic}.

The space of $\Phi^t$-invariant Borel probability
measures for the Weil-Petersson flow can 
naturally be equipped with the weak$^*$-topology.
With respect to this topology, it is a closed convex
set in the topological vector space 
of all finite signed
Borel measures on ${\cal T}^1{\cal M}(S)$. The extreme
points of this convex set are just the ergodic measures.
Thus Theorem 2 
from the introduction is an immediate consequence of the 
following

\begin{proposition}\label{approximation}
Any $\Phi^t$-invariant ergodic Borel probability measure 
can be approximated in the weak$^*$-topology 
by measures supported on closed 
orbits.
\end{proposition}
\begin{proof}
Let $\nu$ be a Borel probability measure on $T^1{\cal M}(S)$ which is
invariant and ergodic for the 
Weil-Petersson geodesic flow $\Phi^t$.
We have to find a sequence of 
periodic orbits for $\Phi^t$ 
so that the normalized 
Lebesgue measures $\nu_i$ supported on these orbits  
converge weakly as $i\to \infty$ to $\nu$. This means
that for every continuous function $f:T^1{\cal M}(S)\to \mathbb{R}$ with 
compact support we have
\[\int fd\nu_i\to \int fd\nu.\]

For this let $v\in T^1{\cal M}$ be
a (typical) density point for the measure $\nu$. 
By the Birkhoff ergodic theorem,
we have 
\[\lim_{t\to \infty} 
\frac{1}{t}\int_0^t f(\Phi^sv)ds= \int fd\nu \]
for every continuous function $f$ on $T^1{\cal M}(S)$ with compact support.
Thus it suffices to find a sequence of 
numbers $t_i\to \infty$ and a sequence of 
periodic orbits 
for $\Phi^t$ which are the supports of normalized $\Phi^t$-invariant 
measures $\nu_i$ such that
\begin{equation}\label{orbitap}
\vert \frac{1}{t_i}\int_0^{t_i}f(\Phi^sv)ds-
\int fd\nu_i \vert \to 0\,(i\to \infty)\notag
\end{equation}
for every continuous function 
$f:T^1{\cal M}(S)\to \mathbb{R}$ with compact support.

The Weil-Petersson metric induces a Riemannian metric
and hence a distance function $d_S$ on (the orbifold)
$T^1{\cal M}(S)$ (the so-called Sasaki metric).
Since a continuous function $f:T^1{\cal M}(S)\to
\mathbb{R}$ with
compact support is bounded and uniformly continuous, for
every $\epsilon >0$ there is a number $\delta >0$ 
depending on $f$ such that
\[\vert\frac{1}{T}\int_0^Tf(\Phi^tw)dt-\frac{1}{T}\int_0^T
f(\Phi^tu)dt\vert < \epsilon\]
whenever $w,u\in T^1{\cal M}(S)$
are such that $d_S(\Phi^tu,\Phi^tw)<\delta$ 
for all $t\in [\delta T,(1-\delta)T]$.

Since the sectional curvature of 
the Weil-Petersson metric is negative, comparison with
the euclidean plane
shows that the Sasaki distance in 
the covering space $T^1{\cal T}(S)$ can geometrically be
estimated as follows. Let $P:T^1{\cal T}(S)\to {\cal T}(S)$
be the canonical projection. Then 
for every $\delta >0$ there
is a number $R=R(\delta)>0$ with the following property.
Let $w,u\in T^1{\cal T}(S)$ be two unit tangent vectors
such that the flow-lines $\Phi^tw,\Phi^tu$ of $w,u$ 
are defined on the interval $[-R,R]$. If 
$d_{WP}(P\Phi^{t}u,P\Phi^{t}w)\leq 1/R$ for all
$t\in [-R,R]$ then $d_S(u,w)<\delta$.

Let $\gamma:\mathbb{R}\to {\cal T}(S)$ be a geodesic whose
initial tangent $\gamma^\prime(0)$ is a preimage of $v$.
By the above discussion and convexity of the distance
function, it suffices to 
find a sequence of 
numbers $t_i\to \infty$ and a sequence $(g_i)\subset {\rm Mod}(S)$ 
of pseudo-Anosov
elements with the following properties.
\begin{enumerate}
\item There is a number $p>0$ such that 
the translation
length of $g_i$ is contained in the interval $[t_i-p,t_i+p]$ 
for all $i$.
\item
For every number $\delta >0$ 
there is a number $T=T(\delta)>0$ not depending on $i$ such that
the distance
between the points $\gamma(T),\gamma(t_i-T)$ and the
axis $\gamma_i$ of $g_i$ is at most $\delta$ 
for all sufficiently large $i$. 
\end{enumerate}

For the construction of such a sequence of pseudo-Anosov
elements, note that
by the Poincar\'e recurrence theorem we may assume 
that there is a sequence
$t_i\to \infty$ such that $\Phi^{t_i}v\to v$ $(i\to \infty)$.
Let as before $\gamma:\mathbb{R}\to {\cal T}(S)$ be a geodesic
whose initial velocity is a preimage of $v$ and 
write $x_0=\gamma(0)$. Then there is a sequence $(g_i)\subset {\rm Mod}(S)$
such that 
\[g_i^{-1}\gamma^\prime(t_i)\to \gamma^\prime(0)\, (i\to \infty).\]
In particular, $d_{WP}(\gamma(t_i),g_ix_0)\to 0$
$(i\to \infty)$.

Let $\zeta_i:[0,r_i]\to {\cal T}(S)$ be the
geodesic connecting $x_0$ to $g_i(x_0)$. 
It follows from the above discussion that 
it suffices to show that for sufficiently large $i$ the 
element $g_i\in G$ is pseudo-Anosov and its axis 
has property 2) above where the 
geodesic arc $\gamma[0,t_i]$ in the statement is
replaced by the geodesic arc $\zeta_i$.

For this recall from Lemma \ref{semisimple2} that 
for each $i$, the element $g_i\in {\rm Mod}(S)$ either is axial or it
is elliptic. If $g_i$ is axial then 
let $\gamma_i:\mathbb{R}\to \overline{{\cal T}(S)}$ be
an oriented axis of $g_i$. 
Let $x_i=\pi_{\gamma_i(\mathbb{R})}(x_0)$.
Then we have $\pi_{\gamma_i(\mathbb{R})}(g_ix_0)=g_ix_i=
\gamma_i(\tau_i)$ where $\tau_i>0$ is the minimum of the
displacement function of $g_i$. Note that
$\tau_i\leq d_{WP}(x_0,g_ix_0)\leq t_i+1$ for all large $i$.
If $g_i$ is not axial then let $x_i$ be a fixed point of $g_i$.

\begin{figure}[ht]
\begin{center}
\psfrag{xi}{$x_i$}
\psfrag{gixi}{$gx_i$}
\psfrag{di}{$\gamma_i$}
\psfrag{ni}{$\eta_i$}
\psfrag{D}{$\Delta_i$} 
\psfrag{Si}{$\zeta_i$}
\psfrag{x0}{$x_0$}
\psfrag{gix0}{$g_ix_0$}
\includegraphics 
[width=0.8\textwidth] 
{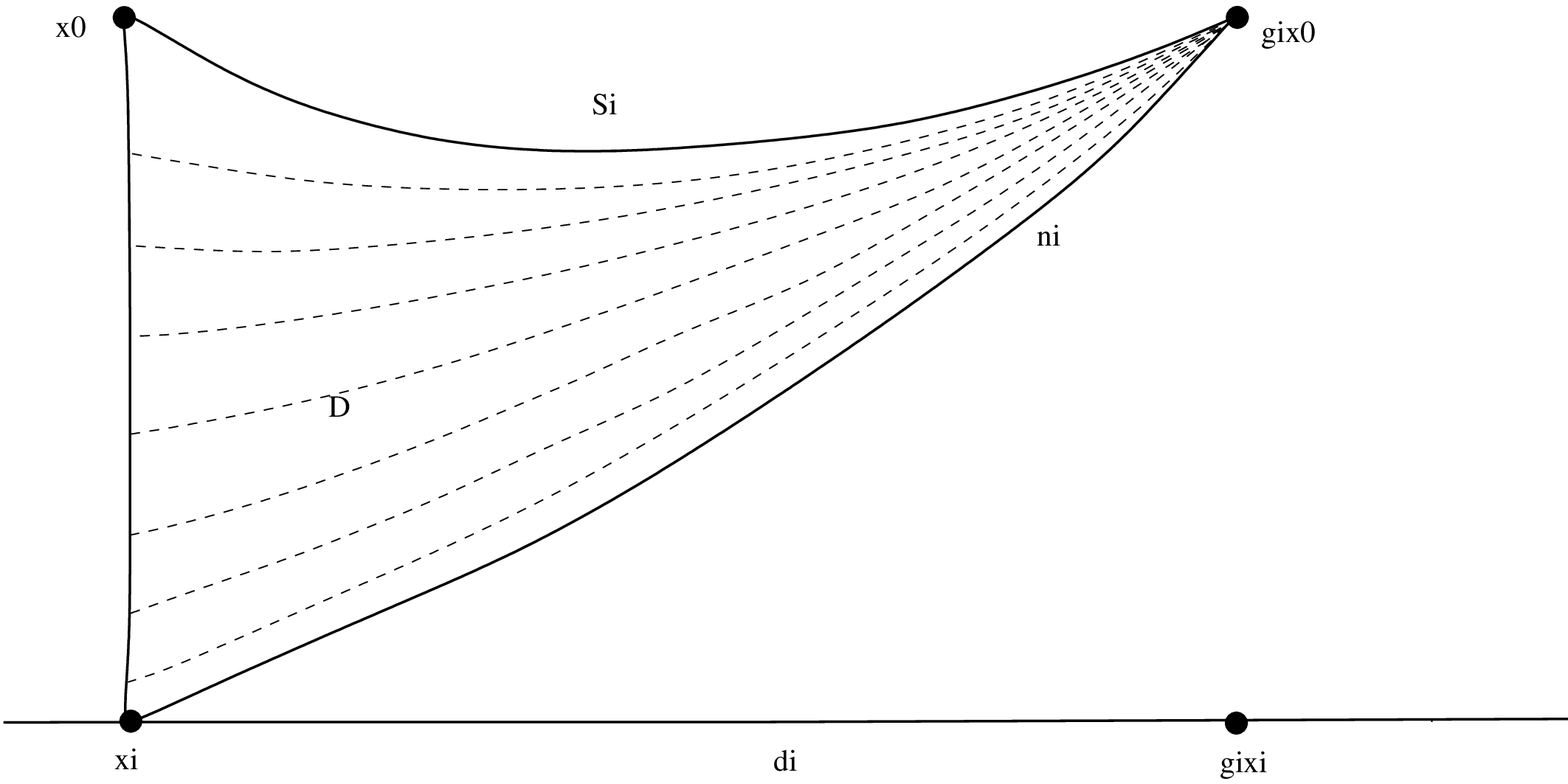}
\end{center}
\end{figure}

Consider the (possibly degenerate) geodesic quadrangle $Q_i$ in 
$\overline{{\cal T}(S)}$ with vertices 
$x_0,x_i,g_ix_i,g_ix_0$. 
The ${\rm CAT}(0)$-angle of $Q_i$ at the vertices $x_i,g_ix_i$
is not smaller than $\pi/2$.
Since $d_{WP}(x_0,x_i)=d_{WP}(g_ix_0,g_ix_i)$, 
by convexity of the distance function the angles of 
$Q_i$ at $x_0,g_ix_0$ do not exceed $\pi/2$. On the other hand,
by equivariance under
the action of $g_i$, the sum of the angles of $Q_i$ at 
$x_0$ and at $g_ix_0$ 
is not smaller than the angle at $x_0$ between the tangent
of the geodesic arc $\zeta_i$ and the negative
$-g_i\zeta_i^\prime(r_i)$ of the
tangent of the geodesic arc $g_i^{-1}\zeta_i$.
Now $\zeta_i^\prime(0)\to \gamma^\prime(0)$, 
$g_i^{-1}\zeta_i^\prime(r_i)\to \gamma^\prime(0)$ and therefore
this angle converges to $\pi$ as $i\to \infty$. Thus
we may assume that 
the minimum of the two angles of $Q_i$ at $x_0,g_ix_0$ 
is bigger than $\pi/4$ for all $i$.

Let $\Delta_i\subset \overline{{\cal T}(S)}$ be the
ruled triangle with vertices $x_0,x_i,g_ix_0$ 
which we obtain by
connecting the vertex $g_ix_0$ to each point on the opposite
side by a geodesic (see the figure).
A geodesic for the Weil-Petersson metric
connecting a point in ${\cal T}(S)$ to a point
in $\overline{{\cal T}(S)}$ is entirely contained
in ${\cal T}(S)$ except possibly for its endpoint \cite{W03,W06}.
Thus $\Delta_i-x_i$ is a smooth embedded surface
in ${\cal T}(S)$ whose intrinsic Gau\ss{} curvature is defined. 
The Gau\ss{} curvature at a point $x\in \Delta_i$ 
does not exceed the maximum of the 
sectional curvatures of the Weil-Petersson metric at $x$ 
(compare the proof of Lemma \ref{fiber} and see \cite{BMM08}
for more and references).

Let $\epsilon>0$ be such that $x_0\in {\cal T}(S)_{2\epsilon}$
and let 
$A\subset T^1{\cal M}(S)$ be the projection
to $T^1{\cal M}(S)$ of the set of all unit tangent vectors in 
$T^1{\cal T}(S)$ with foot-point in ${\cal T}(S)_\epsilon$.
Then $A$ is compact.
Let $\chi$ be the characteristic
function of $A$. Since $v$ is a density point for
$\nu$ by assumption and since there is a neighborhood
of $v$ in $T^1{\cal M}(S)$ which is entirely contained
in $A$, we have
$\int\chi d\nu >0$. Let $r>0$ be sufficiently
small that the sectional curvature of the restriction
of the Weil-Petersson metric to the 
$r$-neighborhood of ${\cal T}(S)_{\epsilon}$
is bounded from above by a negative constant $-b<0$.
Such a number exists by invariance under the action of
the mapping class group and cocompactness.

By ${\rm CAT}(0)$-comparison, there is a number
$s>0$ with the following property. Let $R,S>0$ and let 
$\zeta:[0,R]\to 
\overline{{\cal T}(S)}$, $\eta:[0,S]\to \overline{{\cal T}(S)}$  
be any two geodesics
issuing from $\eta(0)=\zeta(0)=x_0$ which enclose 
an angle at least $\pi/4$. Then
$\eta$ does not pass through the $r$-neighborhood of
$\zeta[s,R]$. 
For this number $s$ and 
for any number $\delta <r/3$ let 
$\tau=\tau(\delta)>0$ be such that
$\chi(\Phi^{\tau}v)=1$ and 
$\int_s^\tau \chi(\Phi^sv)ds\geq \pi/b\delta$.
Such a number exists since by the Birkhoff ergodic theorem
we have 
\[\lim_{t\to \infty}\frac{1}{t}\int_0^t\chi(\Phi^sv)ds=
\int \chi d\nu >0.\]

Let $i>0$ be sufficiently large that $t_i>\tau+2r$,
and that 
$d_{WP}(\zeta_i(r_i),\gamma(t_i))\leq \delta/2$.
By convexity, we then have $d_{WP}(\zeta_i(t),\gamma(t))\leq \delta$ 
for all $t\in [0,r_i]$.
Let $\eta_i$ be the side of $\Delta_i$ connecting $x_i$ to
$g_ix_0$. If $c$ is a geodesic arc in $\Delta_i$ of length
at most $r$ issuing from a point in $\zeta_i[s,R]$ then $c$
does not intersect the side of $\Delta_i$ connecting $x_i$ to 
$x_0$ and hence either it ends on a point in $\eta_i$ or
it can be extended. Thus if the distance between
$\zeta_i(\tau)$ and $\eta_i$ is bigger than
$\delta$ then by convexity 
the ruled triangle $\Delta_i$ contains an embedded strip of 
width $\delta<r/3$ with the arc $\zeta_i[s,\tau]$ as
one of its sides.
This strip is the union of all geodesic arcs of length $\delta$ 
in $\Delta_i$ with respect to the intrinsic metric 
which issue from a point in $\zeta_i[s,\tau]$ and which
are perpendicular to $\zeta_i$. 
If $c$ is such a geodesic arc issuing from a point
$\zeta_i(t)$ where $t\geq s$ is such that 
$\gamma_i^\prime(t)\in A$, then the 
Gau\ss{} curvature of $\Delta_i$ on each point of $c$
does not exceed $-b$. By the choice of $\tau$ and by
volume comparison, the
Lebesgue measure of the set of all points on such geodesic 
arcs is not smaller than $\pi/b$. 
Since the Gauss curvature of 
$\Delta_i$ is negative, this
implies that the
integral of the Gauss curvature of $\Delta_i$ over 
this strip is smaller than $-\pi$. However, this violates   
the Gauss-Bonnet theorem
(compare \cite{H08e} for more details for this
argument, and see also the proof of Lemma \ref{fiber}).

As a consequence, the geodesic $\eta_i$ 
passes through the $\delta$-neighborhood of the point 
$\zeta_i(\tau)$. This implies that for sufficiently large $i$
the isometry $g_i$ is not elliptic. 
Namely, we have $d_{WP}(x_i,\zeta_i(\tau))\geq
d_{WP}(x_0,x_i)-\tau$ and 
$d_{WP}(x_0,g_ix_0)\geq t_i-\delta$ and hence
if $t_i>2\tau+2\delta$ 
then \begin{align}\label{elliptic}
d_{WP}(x_i,g_ix_0) & \geq 
d_{WP}(x_i,\zeta_i(\tau))+t_i-\tau-2\delta \\ \geq 
d_{WP}(x_i,x_0)+t_i-2\tau-2\delta & >
d_{WP}(x_i,x_0).\notag\end{align}
On the other hand, if $g_i$ is elliptic then we have
$d_{WP}(x_i,x_0)=d_{WP}(g_ix_i,g_ix_0)=d_{WP}(x_i,g_ix_0)$
which contradicts inequality (\ref{elliptic}).
Thus $g_i$ is axial for all sufficiently large $i$. 
Let $\gamma_i$ be an oriented axis for $g_i$.

We observe next that $g_i$ is pseudo-Anosov for all $i$.
Namely, using the above notation, let $s_i>0$ be
such that $d_{WP}(\zeta_i(\tau),\eta_i(s_i))<\delta$. 
If the distance between the axis $\gamma_i$ of $g_i$ 
and $\eta_i(s_i)$ is
smaller than $\delta$ then the axis $\gamma_i$ of $g_i$
passes through the $3\delta<r$-neighborhood of a point
in ${\cal T}(S)_\epsilon$
and hence it passes through a point in ${\cal T}(S)$.
As a consequence, $g_i$ is pseudo-Anosov (see the remark 
after Lemma \ref{semisimple2}).

On the other hand,
if the distance between $\gamma_i$ and $\eta_i(s_i)$ is
bigger than $\delta$ then 
let $T>\tau$ be such that $\int_s^T\chi(\Phi^sv)ds\geq 2\pi/b\delta$
and $\chi(\Phi^Tv)=1$. Apply the above
consideration to the 
ruled triangle $\hat\Delta_i$ with vertices $x_i,g_ix_i,
g_ix_0$ which we
obtain by connecting $x_i$ to each  
point on the opposite side by a geodesic segment and to the 
subarc $\eta_i[s_i,s_i+T-\tau]$ 
of $\eta_i$.
We conclude that 
$d_{WP}(\eta_i(s_i+T-\tau),\gamma_i)<\delta$ and once
again, $\gamma_i$ passes through a point in
${\cal T}(S)$ and $g_i$ is pseudo-Anosov.

By the above consideration, 
the axis $\gamma_i$ of $g_i$ passes through
the $\delta$-neighborhood of $\zeta_i(T)$ where
$T=T(\delta)$ only depends on $\delta$. The same
argument shows that this axis also
passes through the $\delta$-neigbhorhood of 
$\zeta_i(r_i-\tilde T)$ where once more $\tilde T>0$ only depends on 
$\delta$ (assume without loss of generality
that $-v$ is a density point for the image
of $\nu$ under the flip $w\to -w$ and use the fact that
two orbit segments of the same finite length are uniformly close
if their initial points are close enough).
In particular, the translation length of $g_i$ is contained
in the interval $[t_i-T-\tilde T,t_i+1]$. 

As a consequence, $(g_i)\subset {\rm Mod}(S)$ is 
a sequence of pseudo-Anosov elements which satisfies
the conditions 1),2) above. Therefore the 
normalized Lebesgue
measures on the projections to $T^1{\cal M}(S)$ of the
unit tangent lines of the axes of the elements
$g_i$ converge weekly to $\nu$. This completes
the proof of the proposition.
\end{proof}

\bigskip

\noindent
MATHEMATISCHES INSTITUT DER UNIVERSIT\"AT BONN\\
BERINGSTRA\SS{}E 1\\
D-53115 BONN\\

\smallskip

\noindent
e-mail: ursula@math.uni-bonn.de

\end{document}